\documentclass[12pt]{amsart}
\usepackage{amsmath,amsfonts,amssymb,verbatim}
\usepackage{amscd,amsthm}
\usepackage{geometry}
\usepackage{enumerate}
\usepackage{color}
\usepackage{mathtools, MnSymbol}

\newtheorem{theorem}{Theorem}

\newtheorem{lemma}[theorem]{Lemma}
\newtheorem{corollary}[theorem]{Corollary}
\theoremstyle{remark}
\newtheorem{remark}[theorem]{Remark}

\theoremstyle{definition}

\newtheorem{example}[theorem]{Example}

\newcommand{\BCH}{\operatorname{BCH}}
\newcommand{\leng}{\ell}
\newcommand{\ptop}{m}
\newcommand{\horbun}{H}
\newcommand{\abel}{^{\mathrm{ab}}}
\newcommand{\deriv}{\mathfrak{d}}
\newcommand{\vf}{\mathfrak{X}}
\newcommand{\dist}{\operatorname{\varrho}}
\newcommand{\distNSW}{\operatorname{\varrho_{\mathrm{NSW}}}}
\newcommand{\afterbar}{\bar{\phantom{x}}}
\newcommand{\fn}[1]{\operatorname{\mathnormal{#1}}}

\newcommand{\dbyd}[3]{\frac{d^{#2} #3}{d#1^{#2}}}
\newcommand{\pdbyd}[3]{\frac{\partial^{#2} #3}{\partial{#1}^{#2}}}
\newcommand{\spdbyd}[3]{\sfrac{\partial^{#2} #3}{\partial{#1}^{#2}}}

\newcommand{\vect}[1]{\smash{\overset{\rightharpoonup}{\rule{0pt}{1.2ex}\smash{{#1}}}}}

\newcommand{\R}{\mathbb{R}}

\newcommand{\Z}{\mathbb{Z}}
\newcommand{\N}{\mathbb{N}}
\newcommand{\Alpha}{\mathrm{A}}

\newcommand{\sfrac}[2]{{#1}/{#2}}

\newcommand{\labs}{\left\vert}
\newcommand{\rabs}{\right\vert}

\newcommand{\lpnorm}{\operatorname{P}\left(}
\newcommand{\rpnorm}{\right)}
\newcommand{\lnorm}{\left\Vert}
\newcommand{\rnorm}{\right\Vert}

\newcommand{\lset}{\left\{}
\newcommand{\rset}{\right\}}

\newcommand{\lpar}{\left(}
\newcommand{\rpar}{\right)}
\newcommand{\Biglpar}{\Bigl(}
\newcommand{\Bigrpar}{\Bigr)}

\renewcommand{\lbrack}{\left[}
\renewcommand{\rbrack}{\right]}
\newcommand{\lip}{\left<}
\newcommand{\rip}{\right>}
\newcommand{\biglpar}{\bigl(}
\newcommand{\bigrpar}{\bigr)}
\newcommand{\rest}[1]{\bigl|_{#1}}

\newcommand{\group}[1]{\mathrm{#1}}
\newcommand{\Lie}[1]{{\mathfrak{#1}}}
\newcommand{\Centre}{\mathfrak{Z}}

\newcommand{\Span}{\operatorname{span}}  
\newcommand{\Exp}{\operatorname{Exp}}

\newcommand{\ad}{\operatorname{ad}}
\newcommand{\Aut}{\operatorname{Aut}}
\newcommand{\Der}{\operatorname{Der}}
\newcommand{\Prol}{\operatorname{Prol}}
\newcommand{\Ad}{\operatorname{Ad}}

\numberwithin{theorem}{section}
\numberwithin{equation}{section}

\begin{document}

\title{Maps of Carnot groups}
\author{Michael G.\ Cowling}
\address{School of Mathematics and Statistics\\ University of New South Wales\\UNSW Sydney 2052\\ Australia}
\author{Alessandro Ottazzi}
\address{CIRM, Fondazione Bruno Kessler\\Via Sommarive 15\\I-38123 Trento\\ Italy}

\title{Global contact and quasi\-conformal mappings \\
of Carnot groups}
\begin{abstract}
 We show that globally defined quasi\-conformal mappings of rigid Carnot groups are affine, but that globally defined contact mappings of rigid Carnot groups need not be quasi\-conformal, and \emph{a fortiori} not affine.
\end{abstract}

\maketitle
\section{Introduction}

Carnot groups are models for sub-Riemannian manifolds, in much the same way as Euclidean spaces are models for Riemannian manifolds.
The study of particular kinds of mappings on Carnot groups is therefore a prelude to the study of the same kinds of mappings on sub-Riemannian manifolds.
For example, the proof that $1$-quasi\-conformal mappings of Carnot groups are automatically smooth led to a proof that isometric mappings of sub-Riemannian manifolds are automatically smooth (see \cite {Capogna-Cowling, Capogna_LeDonne}) and may well lead to the corresponding result for $1$-quasi\-conformal mappings.
In this paper, we study contact and quasi\-conformal mappings of Carnot groups as a step toward understanding the behaviour of contact and quasi\-conformal mappings of sub-Riemannian manifolds.
In particular, we examine the behaviour of globally defined contact and quasi\-conformal mappings on ``rigid'' Carnot groups.
By rigid Carnot group, we mean one for which the space of contact flows is finite-dimensional; by mapping, we always mean a self-mapping.
Our results show that Carnot groups are more varied than Iwasawa $N$ groups, which are the model groups for ``parabolic geometries'' \cite{Cap-Slovak}, and illustrate the fact that sub-Riemmannian manifolds come in many more varieties than parabolic geometries.

The first examples of Carnot groups are the Euclidean space $\R^n$ and the real Heisenberg groups $\group{H}_n(\R)$, which are models for Riemannian geometry and CR geometry respectively.
On the one hand, contact maps are the mappings of CR manifolds that preserve the CR structure; this notion generalises naturally to more general sub-Riemannian geometries.
On the other hand, quasi\-conformal mappings in Euclidean space are now classical and were first studied for Heisenberg groups by A. Kor\'anyi and H. M. Reimann \cite{Koranyi-Reimann1, Koranyi-Reimann2}.
On $\R^n$, the space of quasi\-conformal mappings is infinite-dimensional, while on $\group{H}_n(\R)$ the spaces of contact and of quasi\-conformal mappings are both infinite-dimensional.

Iwasawa $N$ groups, the nilpotent groups that arise in the Iwasawa decomposition $KAN$ of a semi-simple Lie group, are further examples of Carnot groups.
This family of Carnot groups has been studied quite intensively; see, for instance, \cite{CDeKR1,CDeKR2, Freeman0, Ottazzi-Hess, Yamaguchi}.
Other examples of Carnot groups which are reasonably well understood include H-type groups, filiform groups, free nilpotent groups and jet spaces; see, for instance, \cite{Ottazzi-Warhurst-1, Reimann-Htype, Warhurst-filiform, Warhurst-free, Warhurst-jet, Xie-filiform}.

In particular, P. Pansu \cite{Pansu} showed that if there is a locally defined quasi\-conformal mapping between two Carnot groups, then the groups are isomorphic, which reduces the study of mappings between Carnot groups to the study of self-mappings.
He also showed that on the Iwasawa $N$ group associated to $\group{Sp}(n,1)$, smooth contact mappings and not necessarily smooth quasi\-conformal mappings are automatically conformal, and form a finite-dimensional space; this was the key to his celebrated proof of rigidity for  $\group{Sp}(n,1)$.
K. Yamaguchi \cite{Yamaguchi} studied contact mappings on the Iwasawa $N$ groups associated to more general semi-simple Lie groups, and found that for most of these, the space of contact mappings, defined from an arbitrarily small open set  into the whole group, is finite-dimensional.

It is easy to see that, for all rigid Iwasawa $N$ groups, globally defined contact maps are affine; \emph{a fortiori} globally defined quasi\-conformal maps are affine.
A similar result holds for free nilpotent groups \cite{Warhurst-free}, while filiform groups and jet spaces are not rigid.
It therefore seemed plausible that, for rigid Carnot groups, globally defined contact, and \emph{a fortiori} quasi\-conformal, maps must be affine.
In this work, we first give an example of a family of rigid Carnot groups that admit globally defined contact maps which are \emph{not} quasi\-conformal, and then we show that for all rigid Carnot groups, globally defined quasi\-conformal maps \emph{are} affine.

In the rest of this introductory section, we define Carnot groups and contact and quasi\-conformal mappings formally; we explain why globally defined contact mappings of Iwasawa $N$ groups are affine in Section 2.
In Section 3, we exhibit examples of rigid groups for which globally defined contact maps need not be quasi\-conformal, and in Sections 4 and 5 we treat globally defined quasi\-conformal maps.

We point out explicitly here that, while Carnot groups come equipped with a (left-invariant) sub-Riemannian metric, this is not determined uniquely by the stratified group structure, and the space of conformal maps depends on the choice of metric as well as the group.  On the other hand, the spaces of quasi\-conformal and contact mappings only depend on the stratified group structure (though the value of $\lambda$ in $\lambda$ in $\lambda$-quasi\-conformal mapping does depend on the metric).

\subsection{Carnot groups}

We first consider stratified Lie algebras and groups from an algebraic point of view, and then equip a stratified Lie group with the Carnot--Carath\'eodory distance to obtain a Carnot group.
We then define quasi\-conformal maps and prove some preliminary results. Finally, we recall some of the preliminary ideas  of Tanaka prolongation theory.

\subsubsection{Stratified Lie algebras}
Let $\Lie{g}$ be a stratified Lie algebra of step $\leng$.
This means that
\[
\Lie{g}= \Lie{g}_{-1}\oplus \cdots\oplus \Lie{g}_{- \leng},
\]
where $[\Lie{g}_{-j}, \Lie{g}_{-1}] =\Lie{g}_{-j-1}$ when $1\leq j \leq \leng$, while  $\Lie{g}_{-\leng}\neq \{0\}$ and $\Lie{g}_{-\leng-1}=\{0\}$; this implies that $\Lie{g}$ is nilpotent.
We assume that the dimension of $\Lie{g}$ is at least $3$ and finite, to avoid degenerate cases.
Note that the choice of $\Lie{g}_{-1}$ determines the stratification, but the Lie algebra $\Lie{g}$ by itself need not do so.

We write $\Centre(\Lie{g})$ for the centre of $\Lie{g}$; $\pi_k$  for the canonical projection of $\Lie{g}$ onto $\Lie{g}_{-k}$; $d_k$  for $\dim(\Lie{g}_{-k})$; $n_k$ for $d_1 + \dots + d_k$; $n$ for the dimension $\sum_{i=1}^\leng d_j $ of $\Lie{g}$; and $Q$ for the homogeneous dimension $\sum_{i=1}^\leng  j \, d_j $ of $\Lie{g}$.
It is also notationally convenient to set $d_0$ and $n_0$ equal to $0$.
We denote by $\Aut(\Lie{g})$ the group of automorphisms of $\Lie{g}$.
The Lie algebra of $\Aut(\Lie{g})$ is the space of derivations of $\Lie{g}$, which we denote by $\Der(\Lie{g})$.

For each $s \in \R^+$, the dilation $\delta_s \in \Aut( \Lie{g} )$ is defined to be $\sum_{j=1}^{\leng} s^j \pi_j$.
Of course, this definition also makes sense if $s \in \R$, but $\delta_0$ is an endomorphism rather than an automorphism.

For a linear map of $\Lie{g}$, preserving all the subspaces $\Lie{g}_{-j}$ of the stratification is equivalent to commuting with dilations and to having a block-diagonal matrix representation.
We use the adjective 
``strata-preserving'' to describe such maps.
We write $\Aut^\delta(\Lie{g})$ for the subset of $\Aut(\Lie{g})$ of strata-preserving automorphisms and $\Lie{g}_0$ for the space of strata-preserving derivations.

\subsubsection{Stratified Lie groups}
Let $G$ be a stratified Lie group of step $\leng$.
This means that $G$ is connected and simply connected, and its Lie algebra $\Lie{g}$ is stratified with $\leng$ layers.
The identity of $G$ is written $e$, and we view the Lie algebra $\Lie{g}$ as the tangent space at $e$.

Since $G$ is nilpotent, connected and simply connected, the exponential map $\exp$ is a bijection from $\Lie{g}$ to $G$, with inverse $\log$.
We also write $\delta_s$ for the automorphism of $G$ given by $\exp {} \circ {}{\delta_s} \circ {} \log$.
The differential $T \mapsto (T_*)_e$ is a one-to-one correspondence between endomorphisms of $G$  and of $\Lie{g}$, and $T = \exp{} \circ  {T_*}_e \circ {}\log$.

We sometimes use exponential coordinates of the first kind on $G$.
More precisely, we take a basis $\lset X_1, X_2, \dots, X_n\rset$ of $\Lie{g}$ that is adapted to the stratification, by which we mean that $X_{n_{j-1}+1}, \dots, X_{n_j}$ form a basis of the stratum $\Lie{g}_{-j}$ for each $j$, and then associate $\exp(x_1X_1+ x_2X_2 +\dots+ x_nX_n)$ in $G$ to the coordinates $(x_1, \dots, x_n)$ in $\R^n$.
We refer to $x_1$, $x_2$, \dots, $x_{d_1}$ as coordinates of the first layer, and so on.

The Baker--Campbell--Hausdorff formula implies that, for a stratified Lie group of step $\leng$, there is a polynomial $\BCH$ of two variables of degree $\leng$ such that
\[
\exp( X ) \exp ( Y ) = \exp(\BCH(X,Y))
\qquad\forall X, Y \in \Lie{g};
\]
it is known that
\[
\BCH(X, Y) = X + Y + \tfrac{1}{2} [X,Y] + \tfrac{1}{12} [X, [X, Y]] - \tfrac{1}{12} [Y, [Y, X]] + \dots.
\]

A polynomial $p$ on $G$ is said to be \emph{homogeneous} of degree $k$ if $p( \delta_s x) = s^k p(x)$ for all $s \in \R^+$ and all $x \in G$.
Every homogeneous polynomial is a sum of homogeneous monomials, and the homogeneous degree of $x_1^{\alpha_1} \dots x_n^{\alpha_n}$ is $\sum_{k=1}^\leng k \sum_{j = n_{k-1}+1}^{n_k} \alpha_j$.

\subsubsection{Vector fields and flows}
A vector field $\vect V$ on a Carnot group $G$ is said to be \emph{polynomial} if $\vect V = \sum_{i} p_i \tilde X_i$, where the coefficients $p_i$ are polynomials.
Every polynomial vector field may be written as the sum of homogeneous vector fields.
More precisely, we say that a vector field $\vect V$ on $G$ is \emph{homogeneous of degree $k$} if
\[
\vect V(u \circ \delta_s) = s^{-k} (\vect Vu) \circ \delta_s
\quad\forall s \in \R^+ \quad\forall u \in C^\infty(G).
\]
Equivalently,
\[
(\delta_s)_* (\vect V_p) = s^{-k} \vect V_{\delta_s p}
\quad\forall s \in  \R^+  \quad\forall p \in G.
\]
In particular, if $x$ is a coordinate from the $k$th layer, and $X \in \Lie{g}_{-j}$, then the vector fields $x \tilde X$ and $x \breve X$ are homogeneous of degree $k-j$, where $\tilde X$ and $\breve X$  denote
the left-invariant and right-invariant vector fields that coincide with $X$ and $-X$ at the identity.
From the Baker--Campbell--Hausdorff formula, when $n_{k-1} < i \leq n_k$, we may express the left-invariant and right-invariant vector fields $\tilde X_i$ and $\breve X_i$ in exponential coordinates as follows:
\[
\begin{aligned}
\tilde X_i f (x) = \pdbyd {x_i}{}{f(x)} + \sum_{j = n_k+1}^n p_i^j(x) \pdbyd {x_j}{}{f(x)} \\
\breve X_i f (x) =  - \pdbyd {x_i}{}{f(x)} + \sum_{j = n_k+1}^n q_i^j(x) \pdbyd {x_j}{}{f(x)} ,
\end{aligned}
\]
where the polynomials $p_i^j$ and $q_i^j$ are homogeneous and $p_i^j(x) \spdbyd {x_j}{}{}$ and $p_i^j(x) \spdbyd {x_j}{}{}$ are homogeneous of the same degree as $\spdbyd {x_i}{}{}$.

\subsubsection{The horizontal distribution and contact maps}
Throughout, we write $\Omega$ for an arbitrary nonempty connected open subset of $G$.
The differential of a differentiable map $f : \Omega \to G$ is written $f_*$.

We denote by $L_p$ the left translation by $p$ in $G$, that is, $L_p q = pq$ for all $q\in G$.
The subbundle  $\horbun G$ of the tangent bundle $TG$, where $\horbun_p= (L_p)_*(\Lie{g}_{-1})$, is called the \emph{horizontal distribution}.

Using $(L_p)_*$, we may identify $T_e$, the tangent space at $e$ in $G$, itself identified with $\Lie{g}$, with the tangent space $T_p$ at $p$.
Thus if $f$ is differentiable at $p$, then $((L_{f(p)})_*)^{-1}(f_*) (L_p)_*$ is a linear map of $\Lie{g}$.
It is convenient to denote this \emph{linearised total differential} by $\deriv f_p$.

Each $X$ in $\Lie{g}$ induces a left-invariant vector field $\tilde X$,  equal to $(L_p)_*(X)$ at each point $p \in G$.
By definition, the Lie algebra $\tilde{\Lie{g}}$ of all left-invariant vector fields with vector field commutation is isomorphic to $\Lie{g}$, and it inherits the  stratification of $\Lie{g}$.
We write $\vf(G)$ for the space of all $C^1$ vector fields on $G$.
We say that a vector field $\vect V$ lies in the \emph{horizontal subspace} $\vf^H(G)$ of $\vf(G)$ if and only if
\[
{\vect V}_p \in (L_p)_*  \Lie{g}_{-1}
\quad\forall p \in G.
\]

We say that $f \in C^2(\Omega; G)$ is a \emph{contact map} if $f_*$ preserves $\horbun G$; equivalently, $\deriv f_p (\Lie{g}_{-1}) \subseteq  \Lie{g}_{-1}$ for every $p\in \Omega$.
The proof of the following complementary result is immediate.

\begin{lemma}\label{lem:good-automorphisms}
 Suppose that $T$ is a homomorphism of a stratified Lie  group $G$.
 The following are equivalent:
\begin{enumerate}[(i)]
\item
$T$ is a contact map;
\item
$(T_*)_e$ maps $\Lie{g}_{-1}$ into $\Lie{g}_{-1}$;
\item
$(T_*)_e({\Lie{g}}_{-k})  \subseteq \Lie{g}_{-k}$, for each positive integer $k$;
\item
$T$ and $(T_*)_e$ commute with dilations.
\end{enumerate}
\end{lemma}

A flow is a (reasonably smooth) map $(t,x) \mapsto \phi_t (x)$ of $\R \times G$ to $G$, such that
\begin{equation}\label{eq:defn-of-flow}
\phi_s (\phi_t (x)) = \phi_{s+t} (x)
\quad \forall x \in G \quad\forall s, t \in \R.
\end{equation}
To deal with general mappings, we must also deal with  ``local flows''.
A local flow is a map $(t,x) \mapsto \phi_t (x)$ from $\Upsilon$ into $G$, where $\Upsilon$ is an open subset of $\R \times G$ such that $\Upsilon \cap (\R \times\{p\})$ is an open interval containing $0$ for each $p \in G$, and \eqref{eq:defn-of-flow} holds when both sides make sense.

From the theory of ordinary differential equations, a (local) flow $(t,x) \mapsto \phi_t (x)$ on $\Upsilon$ determines a vector field $\vect V$ on $G$ by the formula
\begin{equation}\label{eq:flow}
\vect V u(p) = \frac{d}{dt} u (\phi_t(p)) \rest{t=0} ;
\end{equation}
conversely, given a vector field  $\vect V$ on $G$, then there is a local flow such that \eqref{eq:flow} holds, at least up to the choice of $\Upsilon$.
We will often use the notation $\Exp(t\vect V)$ for this flow $\phi_t$.
If we can find an interval $I$ containing $0$ such that $\Exp(t\vect V)(q) $ is defined for all $t \in I$ and all $q \in G$, then $\Exp(t\vect V)(q) $ may be defined for all $t \in \R$ and all $q \in G$.

Throughout the article we deal with flows $\phi_t$ that are contact diffeomorphisms for every $t\in I$.
A vector field $\vect V$ is said to be \emph{precontact} if the associated flow is contact.
This is equivalent to asking that
\begin{equation}\label{eq:contact-vf}
[\vect V, \vf^H(G)]\subset \vf^H(G).
\end{equation}
Indeed,  let $\phi_t$ be the flow of a precontact vector field $\vect V$.
Then
$
(\phi_t)_* (\tilde X) \in \vf^H(G)
$
for every $X\in \Lie{g}_{-1}$, and therefore, by definition of Lie derivative \cite{Warner},
$$
[\tilde X,\vect V]=\frac{d}{dt}(\phi_t)_* (\tilde X)\Big|_{t=0} \in \vf^H(G)
\quad\forall X\in \Lie{g}_{-1}.
$$

\subsubsection{The Pansu differential}
We recall that a continuous map $f:\Omega \to G$ is \emph{Pansu differentiable} at $p \in \Omega$ if the limit
\[
 \lim_{s\to 0^+} \delta_s^{-1}\circ L_{f(p)}^{-1}\circ f \circ L_{p}\circ \delta_s(q)
\]
exists, uniformly for $q$ in compact subsets of $G$;  if it exists, then it is a strata-preserving homomorphism of $G$, written $Df_p(q)$.
If $f$ is Pansu differentiable at $p$, then $\log{}\circ Df_p \circ {}\exp$ is a Lie algebra homomorphism, written  $df_p$, and
\[
 df_p(X)=\lim_{s\to 0^+} \log {}\circ \delta_s^{-1}\circ L_{f(p)}^{-1}\circ f \circ L_{p}\circ \delta_t \circ \exp (X)
\]
exists, uniformly for $X$ in compact subsets of $\Lie{g}$.
We call  $Df_p$ the Pansu derivative and $df_p$ the Pansu differential of $f$ at $p$.
By construction, both $Df_p$ and $df_p$ commute with dilations, and so in particular, $df_p$ is a strata-preserving Lie algebra homomorphism.

Note that if $T$ is a strata-preserving automorphism of $G$, then its Pansu derivative ${D}T(p)$ coincides with $T$ at every point, and its Pansu differential $dT(p)$ coincides with the Lie differential $\textrm{log} \circ T \circ \textrm{exp}$ at every point.
Thus our notation is a little different from the standard Lie theory notation, but is not  ambiguous.

The next lemma is well known.

\begin{lemma}\label{lem:Pansu-is-part-of-deriv}
Suppose that $f \in C^2( \Omega, G)$.
Then $f$ is Pansu differentiable in $\Omega$ if and only if $f$ is a contact map.
Further, in this case,
\[
df = \sum_{j=1}^{\leng} \pi_j  \deriv f \pi_j.
\]
\end{lemma}

\begin{proof}
This is true in greater generality.
See \cite{Warhurst-2} for the case of $C^1$ maps and \cite{Magnani-Lipschitz} for the case of bi-Lipschitz maps.
\end{proof}

\subsubsection{The sub-Riemannian metric and the associated distance}
We fix a scalar product $\lip\cdot,\cdot\rip$ on $\Lie{g}$ for which the subspaces $\Lie{g}_{-j}$ are pairwise orthogonal, and write $\lnorm \cdot \rnorm$ for the associated norm.
Using the restriction of this scalar product on $\Lie{g}_{-1}$, we define a left-invariant sub-Riemannian metric on $G$ by the formula
\begin{align}\label{scalarprod}
\lip V, W \rip_p = \lip (L_{p^{-1}})_*(V),(L_{p^{-1}})_*(W) \rip  
\quad\forall V, W \in \horbun_p .
\end{align}


The sub-Riemannian metric gives rise to a left-invariant \emph{sub-Riemannian} or  \emph{Carnot--Carath\'eodory} distance function $\dist$ on $G$.
To define this, we first say that a smooth curve $\gamma$ is  \emph{horizontal} if $\dot\gamma(t)\in \horbun_{\gamma(t)}$ for every $t$.
Then we define the distance $\dist(p,q)$ between points $p$ and $q$ by
\[
\dist(p, q) := \inf\int_0^1 \biglpar \lip  \dot\gamma(t), \dot\gamma(t) \rip_{\gamma(t)} \bigrpar^{1/2}  \,dt ,
\]
where the infimum is taken over all (piecewise) smooth horizontal curves $\gamma$ from $[0, 1]$ to $G$ such that $\gamma(0) = p$ and $\gamma(1) = q$.
The distance function is homogeneous, symmetric and left-invariant, that is,
\[
s^{-1} \dist(\delta_s p, \delta_s q)  = \dist(p, q) = \dist(q, p) = \dist(rq, rp)
\quad\forall p, q, r \in G \quad\forall s \in R^+;
\]
in particular, $\dist(p, q) = \dist(q^{-1}p, e)$.
The stratified group $G$, equipped with the distance $\dist$, is known as a \emph{Carnot group}; we usually omit mention of $\dist$.

There is a natural pseudonorm $\lpnorm \cdot \rpnorm$ on the stratified Lie algebra $\Lie{g}$, given by
\[
\lpnorm X \rpnorm = \Biglpar \sum_{k=1}^{\leng}  \lnorm \pi_k X\rnorm ^{2\leng! / k} \Bigrpar ^{1/ 2\leng!}
\quad\forall X \in \Lie{g},
\]
and an associated left-invariant pseudometric $\distNSW$ on $G$, given by
\[
\distNSW(p,q) = \lpnorm \log(q^{-1}p)\rpnorm
\quad\forall p,q \in G.
\]
A. Nagel, E.~M.~Stein, and S. Wainger \cite{nagelstwe} showed that this pseudometric is equivalent to the Carnot--Carath\'eodory metric, in that
\[
C_1 \dist(p,q) \leq \distNSW(p,q) \leq C_2 \dist(p,q)
\quad\forall p,q \in G
\]
(for suitable positive constants $C_1$ and $C_2$).

\begin{lemma}\label{lem:homo-norm}
There exists a constant $C$ such that
\[
\dist( \delta_s p, p ) \leq C \labs s-1\rabs^{1/\leng} \dist( p, e )
\quad\forall p \in G \quad\forall s \in (0,2).
\]
\end{lemma}

\begin{proof}
The result is obvious if $p = e$, and the formula to be proved is homogeneous, so we may assume that $\dist(p,e) = 1$.

Suppose that $p = \exp(X)$, where $X \in \Lie{g}$.
Now
\[
\begin{aligned}
\dist(\delta_s p , p)
&= \dist(p^{-1} \delta_s p, e) \leq C_1^{-1} \distNSW(p^{-1} \delta_s p, e)  \\
&= C_1^{-1}  \Biglpar \sum_{k=1}^{\leng}  \lnorm \pi_k \BCH(-X, \delta_s X) \rnorm ^{2\leng! / k} \Bigrpar^{1/2 \leng!} .
 \end{aligned}
\]
The map $ (X,s) \mapsto \BCH(-X, \delta_s X)$ from $\Lie{g} \times \R$ to $\Lie{g}$ is polynomial, and vanishes when $s =1$;
hence there exists a constant $C_0$ such that
\[
\lnorm \BCH(-X, \delta_s X) \rnorm \leq C_0  \labs s-1\rabs
\]
for all $s \in (0,2)$ and $X$ such that $\dist(\exp(X),e) \leq 1$.
Since $\pi_k$ is a norm-one projection,
\[
\begin{aligned}
\dist(\delta_s p , p)
&\leq C_1^{-1} \Biglpar \sum_{k=1}^{\leng}  \lpar C_0  \labs s-1\rabs\rpar ^{2\leng! / k} \Bigrpar^{1/2 \leng!}
\leq C \labs s-1 \rabs^{1/\leng},
 \end{aligned}
\]
as required.
\end{proof}

\subsection{Quasiconformal maps}
Take $s \in \R^+$.
In a Carnot group, the distortion $H(f,p,s)$ of a map $f: \Omega \to G$ at a point $p \in \Omega$ at scale $s \in \R^+$ is defined by
\[
H(f,p,s) = \frac{\sup\lset \dist( f(x), f(p) ) : x \in G, \dist(x,p) = s \rset }{ \inf\lset \dist( f(x), f(p) ) : x \in G, \dist(x,p) = s \rset } \,.
\]
(In a general space with a distance, inequalities $ \dist(x,p) \leq s$ and $ \dist(x,p) \geq s$ may be needed.)
The map $f$ is $\lambda$-quasi\-conformal if
\[
\limsup_{s \to 0+}  H(f,p,s) \leq \lambda
\quad\forall p \in \Omega.
\]
We say that $f$ is quasi\-conformal if it is $\lambda$-quasi\-conformal for some $\lambda \in \R^+$.

If the map $f$  is smooth, then it is $\lambda$-quasi\-conformal if and only if its Pansu differential $df_p$ is $\lambda$-quasi\-conformal at all $p \in \Omega$.
A consequence of this is that smooth contact maps are automatically locally quasi\-conformal; global conformality is a stronger condition.
Suppose that $\eta: [0, \infty) \to [0, \infty)$ is a (bijective) homeomorphism (necessarily increasing).
We say that $f$ is $\eta$-quasisymmetric if
\begin{equation}\label{eq;defn-qysm}
 \frac{ \dist( f(q_1), f(p) )  }{  \dist( f(q_2), f(p) )  }  \leq \eta \lpar  \frac{ \dist( q_1, p )  }{  \dist( q_2, p )  }  \rpar
 \quad \forall p, q_1, q_2 \in \Omega
\end{equation}
for some $\eta$ (of course, we require that $q_2 \neq p$ so that the denominators are not $0$).

It is easy to see that quasisymmetry implies quasi\-conformality.
In this paper, we are interested in the structure of quasi\-conformal and quasisymmetric maps in the special case where $\Omega = G$.
A deep theorem of J. Heinonen and P. Koskela \cite{Heinonen-Koskela-Carnot, Heinonen-Koskela-CG} (the first paper treats Carnot groups while the second paper treats more general metric spaces) states that for each $\lambda$ and $G$, there exists a homeomorphism $\eta$ such that every $\lambda$-quasi\-conformal map is $\eta$-quasisymmetric.
From the equivalence of quasi\-conformality and quasisymmetry, it follows that the inverse of a quasi\-conformal map is also quasi\-conformal.

It is classical that quasi\-conformal maps on $\R^n$ which fix three points form a normal family.
Our next lemma extends this slightly and into the context of Carnot groups.
For all $r, R \in \R^+$ such that $r < R$, we define the annulus $\Alpha_{r, R}$ to be $\{ p \in G : r \leq \dist(p,e) \leq R \}$.

\begin{lemma}\label{lem:qcf}
Fix $\lambda \in [1, \infty)$ and $p, p'  \in G \setminus\{e\}$.
Suppose that $(f_k: k\in \N)$ is a sequence of globally defined $\lambda$-quasi\-conformal maps such that $f_k(e) = e$ for all $k \in \N$ and $\lim_{k \to \infty} f_k(p) =p'$.
Then for all  $r, R \in \R^+$ such that $r < R$, there exist $r', R' \in \R^+$ such that $r' < R'$ and $f_k(A_{r,R}) \subseteq \Alpha_{r', R'}$, and an increasing homeomorphism $\omega: [0, 2R] \to [0, S]$ (for some positive $S$) such that
\[
\dist( f_k(q_1) , f_k(q_2) ) \leq \omega (\dist( q_1 , q_2 ))
\quad\forall q_1, q_2 \in \Alpha_{r, R}.
\]
Moreover, for all  $r', R' \in \R^+$ such that $r' < R'$, there exist $r'', R'' \in \R^+$ such that $r'' < R''$ and $f_k^{-1} (A_{r',R'}) \subseteq \Alpha_{r'', R''}$, and an increasing homeomorphism $\omega': [0, 2R'] \to [0, S']$ (for some positive $S'$) such that
\[
\dist( f_k^{-1}(q_1) , f_k^{-1}(q_2) ) \leq \omega' (\dist( q_1 , q_2 ))
\quad\forall q_1, q_2 \in \Alpha_{r', R'}.
\]
Finally, $f_k^{-1}(p') \to p$ as $k \to \infty$.
\end{lemma}

\begin{proof}
From the convergence hypothesis,
\[
0 <  \inf\{ \dist( f_k(p), f_k(e) ) : k \in \N \} \leq \sup\{ \dist( f_k(p), f_k(e) ) : k \in \N \}  < \infty.
\]
 From \eqref{eq;defn-qysm} (with the letters changed),
 \begin{equation}\label{eq:qsymineq1}
 \frac{ \dist( f_k(q), f_k(e) )  }{  \dist( f_k(p), f_k(e) )  }  \leq \eta \lpar  \frac{ \dist( q, e )  }{  \dist( p, e )  }  \rpar
\end{equation}
 whence
 \[
 \dist( f_k(q), e ) \leq \dist( f_k(p), e ) \fn\eta \lpar  \sfrac{ R  }{  r }  \rpar  = R',
 \quad
 \]
 say, for all $q \in \Alpha_{r, R}$.
 Similarly,
  \begin{equation}\label{eq:qsymineq2}
 \frac{ \dist( f_k(p), f_k(e) )  }{  \dist( f_k(q), f_k(e) )  }  \leq \eta \lpar  \frac{ \dist( p, e )  }{  \dist( q, e )  }  \rpar
\end{equation}
 whence
 \[
 \dist( f_k(q), e ) \geq \dist( f_k(p), e ) /  \eta \lpar  \sfrac{ R  }{  r }  \rpar  = r',
 \]
say, for all $q \in \Alpha_{r, R}$.

To find $\omega$, observe that from \eqref{eq;defn-qysm},
\[
 \frac{ \dist( f_k(q_1), f_k(q_2) )  }{  \dist( f_k(e), f_k(q_2) )  }  \leq \eta \lpar  \frac{ \dist( q_1, q_2 )  }{  \dist( e, q_2 )  }  \rpar
 \quad \forall q_1, q_2 \in G \setminus \{e\} ,
\]
so $ \dist( f_k(q_1), f_k(q_2) ) \leq R' \fn\eta \lpar  \sfrac{ \dist( q_1, q_2 )  }{  r  }  \rpar $ for all $q_1, q_2 \in \Alpha_{r, R}$, and we may take $\omega(t)$ to be $R' \fn\eta \lpar  \sfrac{ t  }{  r  }  \rpar$.

Next, from \eqref{eq:qsymineq1},
\[
\frac{r'}{ \sup \{ \dist( f_k(p), f_k(e) ) : k \in \N \} }  \leq \frac{ \dist( f_k(q), e )  }{  \dist( f_k(p), e )  }  \leq \eta \lpar  \frac{ \dist( q, e )  }{  \dist( p, e )  }  \rpar ,
\]
whence
\[
\dist( q, e )
\geq \dist( p, e) \fn\eta^{-1} \lpar   \frac{r'}{ \sup \{ \dist( f_k(p), f_k(e) ) : k \in \N \} }   \rpar  = r'',
\]
say, whenever $f_k(q) \in \Alpha_{r',R'}$.
Further, from \eqref{eq:qsymineq2},
\[
\frac{ \inf \{ \dist( f_k(p), f_k(e) ) : k \in \N \} }{R'}   \leq \frac{ \dist( f_k(p), e )  }{  \dist( f_k(q), e )  }  \leq \eta \lpar  \frac{ \dist( p, e )  }{  \dist( q, e )  }  \rpar ,
\]
whence
\[
\dist( q, e )
\leq \dist( p, e) \lbrack \fn\eta^{-1} \lpar   \frac{ \inf \{ \dist( f_k(p), f_k(e) ) : k \in \N \} }{R'}    \rpar \rbrack^{-1} = R'',
\]
say, whenever $f_k(q) \in \Alpha_{r',R'}$.

To find $\omega'$, observe that from \eqref{eq;defn-qysm},
\[
 \frac{  \dist( f_k(e), f_k(q_2) )  }{ \dist( f_k(q_1), f_k(q_2) )  }  \leq \eta \lpar  \frac{  \dist( e, q_2 )  }{ \dist( q_1, q_2 )  }  \rpar
 \quad \forall q_1, q_2 \in G \setminus \{e\} ,
\]
so
\[
 \dist( q_1, q_2 ) \leq R'' \lbrack \fn\eta^{-1} \lpar   \frac{  r'  }{ \dist( f_k(q_1), f_k(q_2) )  }  \rpar \rbrack^{-1}
 \]
for all $q_1, q_2 \in f_k^{-1}(\Alpha_{r', R'})$, and we may take $\omega'(t)$ to be $R'' \lbrack \fn\eta^{-1} \lpar   \sfrac{  r'  }{ t  }  \rpar \rbrack^{-1} $.

Finally, we may choose $r$ and $R$ such that $p \in \Alpha_{r, R}$, so $f_k(p) \in \Alpha_{r', R'}$ for all $k \in\N$.
Then
\[
\dist(f_k^{-1}(p'), p ) = \dist(f_k^{-1}(p'), f_k^{-1}(f_k(p)) ) \leq \omega'(p', f_k(p) ) \to 0  \qquad\text{as $k \to \infty$},
\]
as required.
\end{proof}

 The functions $\omega$ and $\omega'$ in the lemma above are called moduli of continuity.


\subsection{Tanaka prolongation}

The idea behind Tanaka prolongation  is the ``method of flows''.
Tanaka prolongation aims to find the Lie algebra of vector fields satisfying some geometric condition by a purely algebraic process, which amounts to an analysis of the Taylor series of the vector fields.
For example, we might seek to describe the vector fields that generate isometric flows, or conformal flows (see \cite{LeDonne-Ottazzi, Cowling-Ottazzi}), in order to understand isometric or conformal maps by the automorphisms that they induce on the corresponding space of flows by conjugation.

In \cite{Tanaka2}, N.~Tanaka introduced the prolongation $\Prol(\Lie{g},\Lie{g}_0)$ of a stratified Lie algebra through a subalgebra $\Lie{g}_0$ of $\Der^\delta(\Lie{g})$.
The prolongation  has the following properties:
\begin{enumerate}[(P1)]
\item     $\Prol(\Lie{g},\Lie{g}_0)=\sum_{i = -\leng}^{\ptop} \Lie{g}_i$ is a graded Lie algebra and $\Lie{g} = \sum_{i = -\leng}^{-1}\Lie{g}_i$;
\item     if $U\in\Lie{g}_k$ where $k\geq 0$ and $[U,\Lie{g}_{-1}]=0$, then $U=0$;
\item     $\Prol(\Lie{g},\Lie{g}_0)$ is maximal among the Lie algebras satisfying (P1) and (P2).
\end{enumerate}
Tanaka showed that the prolongation is well-defined, and also gave an algorithm to find it.

In (P1), the upper limit $\ptop$ may be a natural number or $+\infty$.
From Tanaka's algorithm, it is clear that all the vector spaces $\Lie{g}_{j}$ (where $j \in \Z$) are finite-dimensional, and so $\Prol(\Lie{g},\Lie{g}_0)$ is finite-dimensional if and only if $m$ is finite.
In this case, we say that the prolongation is finite.


In this paper, we consider the maximal prolongation, where $\Lie{g}_0= \Der^\delta(\Lie{g})$. 
When  $\Prol(\Lie{g}, \Lie{g}_0 )$ is finite, the corresponding Carnot group $G$ is said to be \emph{rigid}.
In this case, there is a graded isomorphism between $\Prol(\Lie{g}, \Lie{g}_0)$ and the Lie algebra of all precontact vector fields, and these are polynomial (see \cite{Yamaguchi}).  More precisely, homogeneous vector fields of degree $k$ correspond  to vectors in $\Lie{g}_k$.
Upper case letters $U$ and $V$ will be used to indicate vectors in $\Prol(\Lie{g},\Lie{g}_0)$, while $\breve U$, $\breve V$ will indicate the corresponding vector fields.
In particular, if $U \in \Lie{g}$, then
\[
(\breve U)u(p) = \frac{d}{dt} u(\exp(-tU)p)\rest{t=0}
\]
for all smooth functions $u$ on $G$.


\section{Carnot groups associated to semisimple Lie groups}
In this short section, to be consistent with the standard notation for semisimple Lie groups, we write $G$ for a semisimple Lie group and $N$ (rather than $G$) for the nilpotent part of a parabolic subgroup.
Let us recall that if $G$ is a noncompact semisimple Lie group, with Cartan involution $\Theta$, then every parabolic subgroup is conjugate to one whose Langlands decomposition  $P = MAN$ satisfies $\Theta(MA) = MA$.
We take such a subgroup $P$, and write $\bar P$ for $\Theta P$.

The Bruhat decomposition of $G$ shows that  $N$ may be identified with the dense open subset $N \bar P$ of the homogeneous space $G/\bar P$.
The maps $n\bar P \mapsto gn \bar P$, where $g \in G$, form an open subset of the set of all contact maps from $N \bar P$ to $G/\bar P$ (there may be issues such as orientability that prevent the $G$ action from exhausting the space of contact maps), unless $G$ is a simple Lie group of real rank one or two (see \cite{Yamaguchi} for a more precise description of the exceptions).
Again from the Bruhat decomposition, the only elements of $G$ that map $N\bar P$ into $N\bar P$ are those in $P$.
So in this case, globally defined contact maps are automatically affine.

A similar discussion holds for multi-contact maps in the higher rank case.
See \cite{CDeKR2}, which addresses the question of whether $G$ exhausts the space of mappings.

\section{Global contact maps}
In this section we establish the existence of global contact maps that are not globally quasi\-conformal on certain rigid Carnot groups.
For nonrigid Carnot groups, the existence of such maps is known; see, for instance, \cite{Koranyi-Reimann1}.

Let $\Lie{g}$ be a $n$-dimensional stratified Lie algebra of step $\leng$ and let $\{ X_1, \dots, X_{n} \}$ be a basis adapted to the stratification.
Denote by $c_{ij}^k$  the structure constants of $\Lie{g}$ with respect to this basis, that is,
\begin{equation}\label{eq:original-structure-constants}
[X_i,X_j]=\sum_k c_{ij}^k X_k.
\end{equation}

Next, consider two copies of $\Lie{g}$, with bases $\{ Y_1,\dots,Y_n \}$ and $\{ Z_1,\dots,Z_n \}$, and denote by $\Lie{h}$ the semidirect product $\Lie{g}\abel  \rtimes \Lie{g}$; this is the vector space with basis $\{ Y_1,\dots,Y_n , Z_1,\dots,Z_n \}$, with the Lie bracket defined by linearity, antisymmetry, and the relations
\begin{equation}\label{eq:semi-direct-structure-constants}
[Y_i,Y_j]= 0 , \qquad [Y_i,Z_j]=\sum_{k=1}^n c_{ij}^k Y_k \qquad\text{and}\qquad [Z_i,Z_j]= \sum_k c_{ij}^k Z_k
\end{equation}
when $1 \leq i,j \leq n$.
Equip the semidirect product $\Lie{h}$ with a stratification as follows.
Set $\Lie{h}_{-k} = \Span\{ Y_i, Z_j : n_{k-1}  < i, j \leq n_k\}$ when $1 \leq k \leq \leng$, and for convenience set $\Lie{h}_{-\leng-1} = \{0\}$.
We may verify by induction that $[\Lie{h}_{-1}, \Lie{h}_{-k}] = \Lie{h}_{-k-1} $ when $1 \leq k \leq \leng$.

The corresponding connected, simply connected and stratified semidirect product group is denoted by $\Lie{g}\abel  \rtimes  G$, or just $H$, and
the canonical projection from $H$ to $G$ is denoted by $\pi$.
Then the typical element of $H$ is written $(Y,z)$, where $Y \in \Lie{g}\abel $ and $z \in G$, and the group operations are given by
\[
\begin{aligned}
(Y, z) (Y', z') = (Y + \Ad(z) Y', zz') , \\
(Y, z)^{-1} = (- \Ad(z)^{-1} Y, z^{-1})
\end{aligned}
\]
Further, $\pi(Y, z) = z$.

Given a $C^1$ vector field $\vect V$ on $G$, we define a vector field $\tau(\vect V)$ on $H$ as follows.
If $\vect V = \sum_{i=1}^n v_i \tilde X_i $, where each $v_i \in C^1(G)$,  then
 \[
 \tau(\vect V) = \sum_{i=1}^n (v_i \circ \pi) \tilde Y_i .
 \]

\begin{lemma}\label{global-contact}
If the vector field $\vect V$ on $G$ is precontact, then the vector field $\tau(\vect V)$ on $H$ just defined is also precontact and its flow is global.
Further, $\deriv (\Exp(t \tau(\vect V)))$ is given by
\[
\begin{aligned}
 \deriv (\Exp(t \tau(\vect V)))(Y) &= Y    \\
 \deriv (\Exp(t \tau(\vect V)))(Z) &= Z +  t  \ad(Z)( W \circ \pi) + t\tilde Z (W \circ \pi)
\end{aligned}
\]
for all $Y \in \Lie{g}$ and all $Z \in \Lie{g}\abel $; here $W$ denotes the $\Lie{g}\abel $-valued function $\sum_{i=1}^n v_i  Y_i$ on $G$.
\end{lemma}

\begin{proof}
Suppose that $1 \leq j \leq d_1$.
Now $\vect V$ is precontact on $G$, so $[{\tilde X}_j , \vect V] \in \vf^H(G)$ by \eqref{eq:contact-vf},  that is,
\[
\sum_{i=1}^n \lpar v_i(x) [ \tilde X_j,  \tilde X_i] + \tilde X_j v_i(x)  \tilde X_i \rpar \in \vf^H(G)
\qquad\forall x \in G.
\]
The structure constants in \eqref{eq:original-structure-constants} and \eqref{eq:semi-direct-structure-constants} coincide and $(v_i \circ \pi)(Y,z) = v_i(z)$, so
\[
\sum_{i=1}^n \lpar (v_i\circ \pi)(Y,z) [ \tilde Z_j,  \tilde Y_i]  + \tilde Z_j (v_i \circ \pi)(Y,z)  \tilde Y_i \rpar \in \vf^H(G)
\qquad\forall (Y,z) \in H;
\]
that is, $[{\tilde X}_j , \tau(\vect V)] \in \vf^H(H)$.
Moreover,
\[
\sum_{i=1}^n \lpar (v_i\circ \pi)(y,z) [ \tilde Y_j,  \tilde Y_i]  + \tilde Y_j (v_i \circ \pi)(y,z)  \tilde Y_i \rpar =0
\qquad\forall (y,z) \in H
\]
since $[ \tilde Y_j,  \tilde Y_i] =0$ and $\tilde Y_j (v_i \circ \pi) = 0$.
Thus $[{\tilde Y}_j , \tau(\vect V)] \in \vf^H(H)$.
We deduce that $\tau(\vect V)$ is precontact.

To compute the flow of $\tau(\vect V)$, we first note that
\[
\dbyd {t}{}{} \pi( \Exp(t \tau(\vect V)) (Y,z)) = 0
\]
since $\pi_* (\tau(\vect V)) = 0$.
Thus we may suppose that
\[
\Exp(t \tau(\vect V))(Y,z) = (Y(t) , z) ,
\]
where $z$ is fixed and $Y(t)$ varies in $\Lie{g}\abel $, and then
\[
\dbyd {t}{}{} \Exp(t \tau(\vect V))(Y,z) = \dbyd {t}{}{} (0,z) (\Ad(z)^{-1} Y(t), 0) . 
\]
We obtain the differential equation
\[
\begin{aligned}
\dbyd t{}  {Y(t)}  &= \Ad(z)  W(z) , 
\end{aligned}
\]
which integrates to
\begin{align*}
Y(t)  &=  t \Ad(z) W(z) + Y(0) .
\end{align*}
In particular, $\Exp(t\tau(\vect V))$ is globally defined for all $t \in \R$.

Now we compute $\deriv \Exp(t\tau(\vect V))$.
On the one hand, if $Y_1 \in \Lie{g}\abel $, then
\[
\begin{aligned}
&(Y(t), z)^{-1} \Exp(t \tau(\vect V)) ( (Y(0),z) (Y_1, e) ) \\
&\qquad = ( t \Ad(z) W(z) + Y(0), z)^{-1} \Exp(t \tau(\vect V)) (Y(0) + \Ad(z) Y_1 ,z)  \\
&\qquad = ( - \Ad(z)^{-1} ( t \Ad(z) W(z) + Y(0))  , z^{-1} )  (t  \Ad(z)W(z) + Y(0) + \Ad(z) Y_1  ,z)  \\
&\qquad =  ( -  ( t W(z) +  \Ad(z)^{-1}Y(0))  + \Ad(z)^{-1}   (t  \Ad(z)W(z) + Y(0) + \Ad(z) Y_1)  , e) \\
&\qquad =  ( Y_1  , e) ,
\end{aligned}
\]
while on the other hand, if $z_1 \in G$, then
 \[
\begin{aligned}
&(Y(t), z)^{-1} \Exp(t \tau(\vect V)) ( (Y(0),z) (0, z_1) ) \\
&\qquad = ( t \Ad(z) W(z) + Y(0), z)^{-1} \Exp(t \tau(\vect V)) (Y(0) , zz_1)  \\
&\qquad = ( - \Ad(z)^{-1} ( t \Ad(z) W(z) + Y(0))  , z^{-1} )  (t  \Ad(zz_1)W(zz_1) + Y(0)  ,zz_1)  \\
&\qquad =  ( -  ( t W(z) +  \Ad(z)^{-1}Y(0))  + \Ad(z)^{-1}   (t  \Ad(zz_1)W(zz_1) + Y(0) )  , z_1) \\
&\qquad =  ( -  t W(z)  - \Ad(z)^{-1}Y(0)  +    t  \Ad(z_1)W(zz_1) + \Ad(z)^{-1} Y(0)   , z_1) \\
&\qquad =  (   t  \Ad(z_1)W(zz_1) - t W(z)  , z_1) .
\end{aligned}
\]
We differentiate with respect to $Y_1$ and $z_1$ to find the derivative.
\end{proof}


\begin{lemma}\label{rigid-semidirect-product}
Let $G$ be a Carnot group.
If $G$ is rigid, then so is $\Lie{g}\abel \rtimes G$.
\end{lemma}
\begin{proof}
By \cite[Theorem 1]{Ottazzi_Warhurst}, a Carnot group $H$ is nonrigid if and only if it satisfies the rank one condition, namely, there exists $X$ in $\Lie{h}_{-1}\otimes_\R {\mathbb C}$ such that $\ad X$ has rank $0$ or $1$, viewed as an endomorphism of the complexification of $\Lie{h}$.
By hypothesis $G$ does not satisfy the rank one condition. We deduce that neither does $\Lie{g}\abel \rtimes G$.

Define the vector space isomorphism $\iota: \Lie{g}\abel \to \Lie{g}$  by linearity and the requirement that $\iota(Y_j)=Z_j$ when $j=1,\dots,n$.
Next, take any vector $W \in \Lie{h}_{-1}$, and suppose that $W=Y+Z$ where $Y\in \Lie{g}\abel $ and $Z\in \Lie{g}$.
Then
\begin{equation}\label{eq:ad-action}
\ad (Y+Z)\Lie{h}= [Y+Z,\Lie{g}\abel +\Lie{g}]=[Y,\Lie{g}]  +  [Z,\Lie{g}\abel ]  +  [Z,\Lie{g}].
\end{equation}
If $Z=0$, then $\ad (Y) \Lie{h}= \ad (Y) \Lie{g}=\iota^{-1}(\ad \iota(Y)\Lie{g})$, where the last equality comes from \eqref{eq:semi-direct-structure-constants}.
This implies that $\Lie{g}\abel \rtimes G$ does not satisfy the rank one condition.
Assume now that $Z\neq 0$.
Since $[Y,\Lie{g}] + [Z,\Lie{g}\abel ] \subseteq \Lie{g}\abel $ and $[Z,\Lie{g}] \subseteq \Lie{g}$, from
\eqref{eq:ad-action} and the hypothesis on $G$ we conclude that $\Lie{g}\abel \rtimes G$ does not satisfy the rank one condition.
\end{proof}

We  now use the results of this section to exhibit examples of rigid Carnot groups that admit global maps which contact but not quasi\-conformal.
Recall that $\Lie{g}_0= \Der^\delta(\Lie{g})$.

\begin{example}
Suppose that $\Lie{g}$ is a stratified Lie algebra, with the property that $\Prol(\Lie{g},\Der^\delta(\Lie{g}))=\sum_{j=-\leng}^m \Lie{g}_j$, where $m$ is finite and $\Lie{g}_1\neq \{0\}$.
Examples of such $\Lie{g}$ are  nilradicals of  minimal parabolic subalgebras of a simple Lie algebra of real rank at least $3$.
By Lemma \ref{rigid-semidirect-product}, $\Lie{g}\abel \rtimes \Lie{g}$ also has finite prolongation.
Take $V\in \Lie{g}_1$ and write the corresponding precontact vector field $ \vect V $ as $\sum_{i=1}^n v_i\tilde X _i$.
Then the coefficients $v_1,\dots,v_{d_1}$ are polynomials of homogeneous degree $2$ \cite{Yamaguchi}, and so their derivatives are unbounded.
By Lemma \ref{global-contact},  $\tau(\vect V)= \sum_{i=1}^n (v_i\circ \pi) \tilde Y_i$ is a precontact vector field in $\Lie{g}\abel \rtimes G$ with a global contact flow.
Moreover,  $\Exp(t \tau(\vect V))$ is not globally quasi\-conformal from the second part of Lemma \ref{global-contact}.
\end{example}

\section{Homogeneous vector fields}

We say that a vector field $\vect V$ on $G$ is  \emph{pre-Lipschitz} or \emph{pre\-quasi\-conformal} if the associated  flow is Lipschitz or quasi\-conformal.
In the case of Lipschitz and quasi\-conformal flows, we require that these flows be global, and that there be a Lipschitz or a quasi\-conformal constant that is valid for all $\Exp(t\vect V)$; the point is that any contact flow is locally Lipschitz and locally quasi\-conformal.
Note that the group property of global flows implies that a Lipschitz flow is actually bi-Lipschitz, that is, Lipschitz with a Lipschitz inverse.

In this section, we prove the following theorem.
Again, recall that $\Lie{g}_0= \Der^\delta(\Lie{g})$.

\begin{theorem}\label{thm:qcf-implies-affine}
Let $\vect V$ be a polynomial pre\-quasi\-conformal vector field on a Carnot group $G$.
Then $\vect V \in \Lie{g}+ \Lie{g}_0 $.
\end{theorem}

\begin{proof}[Outline of proof]
The proof is by contradiction, and is comprised of a series of lemmas.
First, we show that if $\vect V \notin \Lie{g}+ \Lie{g}_0$, then we may assume that $\vect V$ is homogeneous.
Next, we consider the flow $\Exp(\cdot \vect V)$.
We show first
that the trajectories of this flow must leave any annulus $\{ q \in G : r \leq \dist( q, e) \leq R \}$, where $0 < r <R < \infty$, after a finite time.
Next we show that the trajectories cannot tend to $\infty$, but must tend to $0$ as time grows.
Finally, we show that this implies that the trajectories cannot cover $G$, which is absurd.
\end{proof}

\begin{lemma}
 Suppose that $\vect V$ is a polynomial vector field on $G$, that $\vect V = \sum_{j=-\leng}^d {\vect V}^j$ where ${\vect V}^j$ is homogeneous of degree $j$, and that $\Exp(\cdot \vect V)$ is defined globally.
Then
\[
\delta_s \Exp( t s^d \vect V) \delta_{s^{-1}} = \Exp( t \sum_{j= -\leng}^d s^{d-j} {\vect V}^j)
\quad\forall s \in \R^+ \quad\forall t \in \R.
\]
\end{lemma}

\begin{proof}
Observe that
 \[
\begin{aligned}
 \frac{d}{dt} \Exp( t \sum_{j= -\leng}^d s^{d-j} {\vect V}^j)  (q)\rest{t=0}
& =  \sum_{j= -\leng}^d s^{d-j} ({\vect V}^j)_q  \\
&=  s^d \sum_{j= -\leng}^d  (\delta_s)_* ({\vect V}^j)_{\delta_{s^{-1}} q} \\
&= s^d (\delta_s)_* (\vect V)_{\delta_{s^{-1}} q}
 \end{aligned}
 \]
Both sides of the expression to be proved are flows (as functions of $t$), and the result follows.
\end{proof}

\begin{lemma} \label{lem:V-homogeneous}
 Suppose that $\vect V$ is a polynomial vector field on a Carnot group $G$, and that $\vect V = \sum_{j=-\leng}^d {\vect V}^j$ where ${\vect V}^j$ is homogeneous of degree $j$.
If $\Exp(\cdot \vect V)$ is global and $\lambda$-quasi\-conformal , so is $\Exp(\cdot {\vect V}^d)$. 
\end{lemma}

\begin{proof}
Evidently,  $\delta_s \Exp( t s^d \vect V) \delta_{s^{-1}}$ and hence also $\Exp(t \sum_{j= -\leng}^d s^{d-j} {\vect V}^j)$ are $\lambda$-quasi\-conformal for all $s \in \R^+$ and $t \in \R$.
From the theory of ordinary differential equations, $\Exp(t \sum_{j= -\leng}^d s^{d-j} {\vect V}^j)(q)$ converges locally uniformly to $\Exp(t {\vect V}^d)(q)$ as $s \to \infty$.
Hence $\Exp(t {\vect V}^d)(q)$ is $\lambda$-quasi\-conformal.
\end{proof}

\begin{remark}
 We use this argument because it is not clear that the set of pre\-quasi\-conformal vector fields forms a Lie algebra.
 If we were dealing with a different type of map for which the corresponding family of vector fields formed a Lie algebra, then it would follow from the fact that $\Exp(t \sum_{j= -\leng}^d s^{d-j} {\vect V}^j)(q)$ has the desired property for infinitely many values of $s$ that $\Exp(t {\vect V}^j)(q)$ does too for all $j$.
\end{remark}

Recall that $\Alpha_{r,R}$ denotes the annulus $\{ q \in G : r \leq \dist( q, e) \leq R \}$.

\begin{lemma}\label{lem:nice-orbits}
 Suppose that $\vect V$ is a polynomial vector field on a Carnot group $G$, that $\vect V$ is homogeneous of strictly positive degree $j$, and $\Exp(\cdot \vect V)$ is a $\lambda$-quasi\-conformal global flow.
If there exist $p, p' \in G \setminus \{e\}$ and a real sequence $(t_k: k \in \N)$ such that $t_k \to \infty$  and $\Exp(t_k \vect V)(p) \to p'$ as $k \to \infty$, then $\Exp(t \vect V)(p) = p$ for all $t \in  \R$.
\end{lemma}

\begin{proof}
We apply Lemma \ref{lem:qcf}, taking $f_k$ to be $\Exp(t_k\vect V)$.
We may suppose that $\dist( \Exp(t_k \vect V)(p), p') < \frac{1}{2} \dist(e, p')$ for all $k \in \N$.
By the lemma, the maps $\Exp(t_k \vect V)$ are continuous on the annulus $\Alpha_{r, R}$, where $r = \frac{1}{2} \dist(p,e) $ and $R = 2 \dist(p,e) $, with a common modulus of continuity $\omega$, and and they all map $\Alpha_{r, R}$ into an annulus $\Alpha_{r', R'}$.
Moreover, the maps $\Exp(-t_k \vect V)$ are continuous on the annulus $\Alpha_{r', R'}$, with a common modulus of continuity $\omega'$.
Furthermore, $\Exp(- t_k \vect V)(p') \to p$ as $k \to \infty$.

Take $s$ very close to $1$, and observe that
\[
\begin{aligned}
\dist( \Exp( s^j t_k \vect V) (p), p')
 &= \dist( \delta_{s^{-1}} \Exp( t_k \vect V) (\delta_s p), p') \\
 &\leq  \dist( \delta_{s^{-1}} \Exp( t_k \vect V) (\delta_s p), \Exp( t_k \vect V) (\delta_s p) ) \\
 &\qquad +  \dist( \Exp( t_k \vect V) (\delta_s p), \Exp( t_k \vect V)  (p)) \\
 &\qquad +  \dist( \Exp( t_k \vect V)  (p), p') \\
 &\leq C \labs s-1 \rabs^{1/\leng} \dist(  \Exp( t_k \vect V) (\delta_s p), e) \\
 &\qquad + \omega( \dist ( \delta_s p, p) ) \\
 &\qquad + \dist( \Exp( t_k \vect V)  (p), p') ,
\end{aligned}
\]
by Lemma \ref{lem:homo-norm}.
We may make the first and second terms arbitrarily small, uniformly in $k$, by taking $s$ close to $1$; the third term may be made arbitrarily small by taking $k$ large.
We deduce that for all positive $\epsilon$,  we may take $I = t_k (1-\zeta, 1+ \zeta)$ for small $\zeta$ and large $k$, and then $\dist( \Exp( t \vect V) (p), p') < \epsilon$ for all $t \in I$; the intervals $I$ may be made arbitrarily long in $\R$.
Exchanging the role of $p$ and $p'$, we may make $\dist( \Exp( - t \vect V) (p'), p) < \epsilon$ for all $t \in I$ in the same way.
Hence we may make $\dist( \Exp( - t_1 \vect V) \circ \Exp(  t_2 \vect V) (p), p) < 2\epsilon$ for all $t_1$ and $t_2$ in an arbitrary long interval $I$, and so $\dist( \Exp( (t_2- t_1) \vect V) (p), p) < \epsilon$ for all such $t_1$ and $t_2$.
It follows that $\dist( \Exp( (t \vect V) (p), p) < 2\epsilon$ for all $t \in \R$, and since $\epsilon$ is arbitrary, $\Exp( (t \vect V) (p) = p$ for all $t$.
\end{proof}

\begin{lemma}
 Suppose that $\vect V$ is a nonzero polynomial vector field on a Carnot group $G$, that $\vect V$ is homogeneous of strictly positive degree $j$, and $\Exp(\cdot \vect V)$ is a $\lambda$-quasi\-conformal global flow.
Then the only point $p$ in $G$  at which $\vect V_p = 0$ is $e$.
\end{lemma}

\begin{proof}
Suppose that $p \in G \setminus\{ e\}$ and $\vect V_p = 0$.
Take a closed ball $B(p, \epsilon)$ around $p$.
Since $\vect V$ is homogeneous,  $\vect V_{q} =0$ for all $q \in  \{ \delta_s p : s \in \R^+ \}$.
There exists a point $p_\epsilon$ in $\{ \delta_s p : s \in \R^+ \}$ such that $\dist(p, p_\epsilon) = \epsilon$.
Now  $\Exp(t \vect V)$ is $\eta$-quasisymmetric for some homeomorphism $\eta: [0,\infty) \to [0,\infty)$, and so for any $q \in B(p, \epsilon)$,
\[
\frac{ \dist( \Exp(t\vect V)(q), \Exp(t\vect V)(p) ) }{ \dist( \Exp(t\vect V)(p_\epsilon), \Exp(t\vect V)(p) ) } \leq \eta \lpar \frac{ \dist( q, p ) }{ \dist( p_\epsilon, p ) } \rpar,
\]
that is, $\dist( \Exp(t\vect V)(q), p) \leq \epsilon \fn\eta ( 1)$.
If $\Exp(t\vect V)(q) \neq q$ for some $t \in \R$ and $q \in B$, then $\{ \Exp(t \vect V) q : t \in \R \}$ has a limit point, necessarily in $B(p, \epsilon)$, which contradicts the previous lemma.
Otherwise, $\vect V = 0$ in $B(p, \epsilon)$ and so $\vect V = 0$ everywhere.
\end{proof}

\begin{lemma}\label{lem:time-in-ring}
Suppose that $\vect V$ is a nonzero polynomial vector field on a Carnot group $G$, that $\vect V$ is homogeneous of strictly positive degree $j$, and $\Exp(\cdot \vect V)$ is a $\lambda$-quasi\-conformal global flow.
Suppose also that $0 < r < R < \infty$ .
For each $p \in  \Alpha_{r,R}$, let
\[
t_p = \sup\{ t \in [0, \infty) : \Exp(t \vect V)(p) \in \Alpha_{r,R} \} .
\]
For each $p \in  \Alpha_{r,R}$,  the supremum is attained and is finite, and $\sup\{ t_p : p \in \Alpha_{r,R} \}$ is finite.
\end{lemma}

\begin{proof}
First, fix $p \in \Alpha_{r,R}$.
Clearly  $t_p$ is finite, for otherwise there exist $p' \in \Alpha_{r,R}$ and a sequence $(t_k : k \in \N)$ in $\R^+$ such that $t_k \to \infty$ and $\Exp(t_k \vect V)(p) \to p' \in \Alpha_{r,R}$, which contradicts Lemma \ref{lem:nice-orbits}.
Further, if $(t_k : k \in \N)$ is a sequence in $\R^+$ such that  $\Exp(t_k \vect V)(p) \in \Alpha_{r,R}$ and $t_k \to t_p$, then by extracting a subsequence if necessary, we may suppose that $(\Exp(t_k \vect V)(p) : k \in \N)$ converges, whence $\Exp(t_p \vect V)(p) \in \Alpha_{r,R}$.

Now suppose that there exist sequences $(p_k : k \in \N)$  in $\Alpha_{r,R}$ and $(t_k : k \in \N )$ in $\R^+$ such that $t_k \to \infty$ as $k \to \infty$ and $ \Exp(t_k \vect V)(p_k) \in\Alpha_{r,R}$.
By passing to a subsequence if necessary, we may suppose that $p_k \to p$ as $k \to \infty$, where $p \in \Alpha_{r,R}$.

As argued in the proof of Lemma \ref{lem:qcf}, the maps $ \Exp(t_k \vect V)$ are equicontinuous on $\Alpha_{r,R}$ with a common modulus of continuity $\omega$.
Hence
\[
\begin{aligned}
\dist( \Exp(t_k \vect V)(p), \Alpha_{r,R} )
&\leq \dist( \Exp(t_k \vect V)(p), \Exp(t_k \vect V)(p_k) ) + \dist( \Exp(t_k \vect V)(p_k), \Alpha_{r,R} ) \\
& \leq \omega( \dist( p, p_k) ) \to 0
\end{aligned}
\]
as $k \to \infty$, and so the sequence $( \Exp(t_k \vect V)(p) : k \in \N)$ has a subsequence that converges to $p' \in \Alpha_{r,R}$, which contradicts Lemma \ref{lem:nice-orbits}.
\end{proof}

\begin{remark}\label{rem:bla}
The quantities $t_p$ and $\sup\{t_p : p \in \Alpha_{r,R}\}$ of the previous lemma depend on $r$ and $R$.
Further, $\dist( \Exp(t_p \vect V)(p), e)$ is equal to either $r$ or $R$.
\end{remark}

We write $T_k$ for $\sup\{t_p : p \in \Alpha_{2^k, 2^{k+1}} \}$, where $k \in \Z$.

\begin{lemma}\label{lem:which}
 Suppose that $\vect V$ is a nonzero polynomial vector field on a Carnot group $G$, that $\vect V$ is homogeneous of degree $j$, and $\Exp(\cdot \vect V)$ is a $\lambda$-quasi\-conformal global flow.
 Then $T_k = 2^{-kj} T_0$.
\end{lemma}

\begin{proof}
The dilation $\delta_{2^k}$ maps $\Alpha_{1,2}$ onto $\Alpha_{2^k, 2^{k+1}}$, and trajectories of $\Exp(\cdot \vect V)$ in the first annulus onto trajectories in the other.
More precisely,
\[
\Exp(2^{kj} t\vect V) (\delta_{2^k} p) = \delta_{2^k} \Exp(t\vect V)(p)
\quad\forall p \in \Alpha_{1,2}.
\]
Each trajectory in $\Alpha_{2^k, 2^{k+1}}$ is traversed $2^{kj}$ times faster than the corresponding trajectory in $\Alpha_{1,2}$, and the lemma follows.
\end{proof}

\begin{lemma}\label{lem:trajectories-to-0}
 Suppose that $\vect V$ is a nonzero polynomial vector field on a Carnot group $G$, that $\vect V$ is homogeneous of strictly positive degree $j$, and $\Exp(\cdot \vect V)$ is a $\lambda$-quasi\-conformal global flow.
 Then $\lim_{t \to \pm \infty} \Exp(t\vect V)(p_0) = e$ for all $p_0 \in G \setminus\{e\}$.
\end{lemma}

\begin{proof}
By dilating if necessary, we may suppose that $\dist(p_0, e) =1$.

Much as in Lemma \ref{lem:time-in-ring}, define
\[
t_0 = \sup\{ t \in [0, \infty) : \Exp(t\vect V)(p_0) \in \Alpha_{1,2} \}.
\]
Let $p_1 = \Exp(t_0\vect V)(p_0)$, and then $\dist(p_1 ,e)$ is either
$1$ or $2$.

Suppose first that $\dist(p_1 ,e) = 2$, so that $p_1 \in \Alpha_{2,4}$.
Define
\[
t_1 = \sup\{ t \in [0, \infty) : \Exp(t\vect V)(p_1) \in \Alpha_{2,4} \},
\]
and $p_2 = \Exp(t_1\vect V)(p_1)$.
Now $\Exp(t\vect V)(p_1)   = \Exp((t_0+t)\vect V)(p_0) > 2$ if $t > 0$, by definition of $t_0$ and the continuity of $t \mapsto \dist(\Exp(t\vect V)(p_0), e)$.
Hence $\dist(p_2, e) > 2$.
By Remark \ref{rem:bla}, $\dist( p_2  ,e) =2$ or $\dist( p_2  ,e) =4$, and the first option is impossible, so $\dist( p_2  ,e) =4$.
We define recursively
\[
t_{k} = \sup\{ t \in [0, \infty) : \Exp(t\vect V)(p_k) \in \Alpha_{2^k,2^{k+1}} \}
\]
and $p_{k+1} = \Exp(t_{k}\vect V)(p_{k})$; by a similar argument, $\dist( p_k  ,e) = 2^k$ and if $t > 0$, then $\dist(\Exp(t\vect V)(p_k), e) > 2^{k}$.
We have shown that if $t > t_0 + t_1 + \dots + t_k$, then $\dist(\Exp(t\vect V)(p_0),e) > 2^{k+1}$,but
\[
t_0 + t_1 + \dots + t_n \leq T_0 + T_1 + \dots + T_n < T_0 \lpar \frac{1}{1-2^{-l}} \rpar.
\]
Hence $\sum_{k=0}^\infty t_k$ is finite, and  $\Exp(t\vect V)(p_0) \to \infty$ as $t \to \sum_{k=0}^\infty t_k$.
By hypothesis, the flow $\Exp(\cdot \vect V)$ is global, so it cannot be true that $\dist(p_1 ,e) = 2$.

Consequently, $\dist(p_1, e) = 1$.
A similar argument shows that $\Exp(t\vect V)(p_0)$ eventually  lies inside $\Alpha_{2^k, 2^{k+1}}$ for arbitrary $k \in -\N$, and hence $\lim_{t \to \infty} \Exp(t\vect V)(p_0) =e$.

The analysis of what happens when $t \to -\infty$ is analogous.
\end{proof}

\begin{lemma}\label{lem;trajectories-bounded}
 Suppose that $\vect V$ is a nonzero polynomial vector field on a Carnot group $G$, that $\vect V$ is homogeneous of strictly positive degree $j$, and $\Exp(\cdot \vect V)$ is a $\lambda$-quasi\-conformal global flow.
 Then the set $\bigcup_{p \in \Alpha_{1,2}} \Exp(\R \vect V)(p)$ is bounded.
\end{lemma}

\begin{proof}
As before, for $p \in \Alpha_{1,2}$, we define
 \[
 t_p = \sup\{ t \in [0, \infty) : \Exp(t \vect V)(p) \in \Alpha_{1,2} \} ,
 \]
 and $T_0 = \sup\{ t_p : p \in \Alpha_{1,2}\}$.

 Take $q \in \Alpha_{1,2}$.
 If $t \notin [-T_0, T_0]$, then $\Exp(t \vect V)(q) \notin \Alpha_{1,2} $.
 It cannot be true that $\dist(\Exp(t \vect V)(q),e) > 2$ for any such $t$, since $\lim_{t \to \pm\infty} \Exp(t\vect V)(q) = e$.
 Hence
\begin{equation}\label{eq:long-time}
 \dist(\Exp(t \vect V)(q),e) < 1
 \quad\forall t \in \R \setminus [-T_0, T_0].
\end{equation}
Fix $p \in \Alpha_{1,2}$.
Since $\lim_{t \to \pm\infty}\Exp(t\vect V)(p) = e$ and the flow is continuous, the trajectory $\Exp(\R \vect V)(p)$ is bounded, that is, there exists $D_p \in \R^+$ such that
 \[
 \dist(\Exp(t\vect V)(p), e) \leq D_p
 \quad\forall t \in \R.
 \]
By the theory of ordinary differential equations, there exists $\delta_p \in \R^+$ such that if $\dist (q, p) < \delta_p$, then $\dist(\Exp(t\vect V)(q), \Exp(t\vect V)(p)) < 1$ for all $t \in [-T_0, T_0]$.
Thus for these $q$,
\[
 \dist(\Exp(t\vect V)(q), e) \leq D_p +1
 \quad\forall t \in [-T_0, t_0].
 \]
 However the same inequality holds when $t \notin [-T_0, T_0]$, by \eqref{eq:long-time}.

Let $\Omega_p = \{ q \in \Alpha_{1,2} : \dist(p,q) < \delta_p\}$.
The sets $\Omega_p$ form an open cover of $\Alpha_{1,2}$ in the relative topology, as $p$ varies over the compact set $\Alpha_{1,2}$.
Thus there is a finite subcover, $\{ \Omega_{p_1}, \dots, \Omega_{p_n} \}$ say.
It follows that
\[
 \dist(\Exp(t\vect V)(q), e) \leq \max\{ D_{p_1}, \dots , D_{p_n} \} +1
 \quad\forall t \in \R. \quad\forall q \in \Alpha_{1,2}.
\]
This is what we needed to prove.
\end{proof}

\begin{proof}[Proof of Theorem \ref{thm:qcf-implies-affine}]
 Suppose that $\vect V$ is a polynomial pre\-quasi\-conformal vector field on $G$.
 By Lemma \ref{lem:V-homogeneous}, we may suppose that $\vect V$ is homogeneous of degree $j$, where $j$ is the highest degree  of a homogeneous component of $\vect V$.
 If $j > 0$, then every trajectory of $\vect V$ starts and ends at $e$, by Lemma \ref{lem:trajectories-to-0}.
 Further, the trajectories that pass through $\Alpha_{1,2}$ are uniformly bounded, by Lemma \ref{lem;trajectories-bounded}.
 The trajectories through a point outside $\bigcup_{p \in \Alpha_{1,2}} \Exp(\R \vect V)(p)$ cannot pass through $\Alpha_{1,2}$ but have to tend to $e$, and we have a contradiction.
 Hence $j \leq 0$.
\end{proof}


\begin{corollary}
 Suppose that $G$ is a rigid Carnot group.
Then the set of pre\-quasi\-conformal vector fields on $G$ is a Lie algebra, isomorphic to  $\Lie{g}+ \Lie{g}_0$.
\end{corollary}

\begin{proof}
Every pre\-quasi\-conformal vector field is precontact and hence polynomial.
The result follows immediately.
\end{proof}

\section{Global quasi\-conformal maps on rigid Carnot groups}

In this section, we state and prove our characterization of global quasi\-conformal mappings of rigid Carnot groups, which is a consequence of Theorem \ref{thm:qcf-implies-affine}.
Again, recall that $\Lie{g}_0= \Der^\delta(\Lie{g})$.

We recall that the nilradical is the maximal nilpotent ideal of a Lie algebra.
Every  Lie algebra homomorphism sends the nilradical into the nilradical, so every Lie algebra automorphism preserves the nilradical.

\begin{theorem}\label{global-qc-maps}
Let $G$ be a rigid Carnot group, and let $f:G \rightarrow G$ be a $C^2$ map.
Then $f$ is global $\lambda$-quasi\-conformal if and only if $f$ is the composition of a left trans\-lation and a strata-preserving automorphism.
\end{theorem}

\begin{proof}
As it is evident that affine maps are quasi\-conformal, we prove only the forward implication.
Suppose that $f: G \to G$ is $\lambda$-quasi\-conformal.
By composing with a left translation, we may suppose that $f(e) = e$; we must show that $f$ is a strata-preserving automorphism.

Each pre\-quasi\-conformal vector field $\breve U$ on $G$ induces a flow $\Exp(\cdot \breve U)$ on $G$; we define a new quasi\-conformal flow $\psi_t$ on $G$ by conjugation
\begin{equation}\label{eq:conjugation-by-f}
\psi_t = f \circ \Exp(t \breve U) \circ f^{-1}.
\end{equation}
Differentiation with respect to $t$ yields a new pre\-quasi\-conformal vector field $\breve V$ on a neighbourhood of $e$ such that $\psi_t = \Exp(\cdot \breve V)$.
Further, if $\Exp(t \breve U)$ fixes $e$ for all $t$, then so does $\Exp(t\breve V)$.

We write $\breve f$ for the map $\breve U  \mapsto \breve V$, and also for the corresponding map $U \mapsto V$ of $\Lie{g} + \Lie{g}_0$.
Then $\breve f: U \mapsto V$ is an automorphism of the algebra $\Lie{g} + \Lie{g}_0$, and leaves invariant the subalgebra $\Lie{g}_0$ corresponding to vector fields that vanish at $e$.
Being an automorphism, it also leaves invariant the nilradical $\Lie{g}$.
From \eqref{eq:conjugation-by-f}, we see that
\begin{equation}\label{eq:breve-f}
\Exp(t {\breve f(\breve U)}) (e) =  f( \Exp(t \breve U) (e))
\quad\forall t \in \R.
\end{equation}
Clearly there is an automorphism $T$ of $G$, not necessarily strata-preserving, such that $dT = \breve f$.

Now $\Exp(t\breve{W})(e) = \exp(-t W)$ for all $W \in \Lie{g}$.
From \eqref{eq:breve-f}, it follows that
\[
 f(\exp(-t U)) = \exp(-t \breve f (U)) = T \exp(-t U)
\]
for all $U$ in $\Lie{g}$ and all $t \in  \R$, that is, $f$ is an automorphism.
Since $f$ is also a contact map, the automorphism is strata-preserving, by Lemma \ref{lem:good-automorphisms}. 
\end{proof}

\begin{corollary}
 Let $G$ be a rigid Carnot group, and let $f:G \rightarrow G$ be a $C^2$ map.
Then $f$ is global bi-Lipschitz if and only if $f$ is the composition of a left trans\-lation and a strata-preserving automorphism.
\end{corollary}

\bibliography{Biblio-Global-QC}

\providecommand{\bysame}{\leavevmode\hbox to3em{\hrulefill}\thinspace}
\providecommand{\MR}{\relax\ifhmode\unskip\space\fi MR }
\providecommand{\MRhref}[2]{%
  \href{http://www.ams.org/mathscinet-getitem?mr=#1}{#2}
}
\providecommand{\href}[2]{#2}
\begin{thebibliography}{10}

\bibitem{Cap-Slovak}
A.~{\v{C}}ap and J.~Slov{\'a}k, \emph{Parabolic {G}eometries. {I}: Background
  and general theory}, Mathematical Surveys and Monographs, vol. 154, American
  Mathematical Society, Providence, RI, 2009.

\bibitem{Capogna-Cowling}
L.~Capogna and M.~G. Cowling, \emph{Conformality and {$Q$}-harmonicity in
  {C}arnot groups}, Duke Math. J. \textbf{135} (2006), no.~3, 455--479.

\bibitem{Capogna_LeDonne}
L.~Capogna and E.~Le~Donne, \emph{Smoothness of subriemannian isometries},
  submitted (2013).

\bibitem{CDeKR1}
M.~G. Cowling, F.~De~Mari, A.~Kor{\'a}nyi, and H.~M. Reimann, \emph{Contact and
  conformal maps on {I}wasawa {$N$} groups}, Atti Accad. Naz. Lincei Cl. Sci.
  Fis. Mat. Natur. Rend. Lincei (9) Mat. Appl. \textbf{13} (2002), no.~3--4,
  219--232, Harmonic analysis on complex homogeneous domains and Lie groups
  (Rome, 2001).

\bibitem{CDeKR2}
\bysame, \emph{Contact and conformal maps in parabolic geometry. {I}}, Geom.
  Dedicata \textbf{111} (2005), 65--86.

\bibitem{Cowling-Ottazzi}
M.~G. Cowling and A.~Ottazzi, \emph{Conformal maps of {C}arnot groups}, to
  appear in Ann. Acc. Sci. Acc. Fenn.

\bibitem{Freeman0}
D.~M. Freeman, \emph{Invertible {C}arnot groups}, Preprint.

\bibitem{Heinonen-Koskela-Carnot}
J.~Heinonen and P.~Koskela, \emph{Definitions of quasiconformality}, Invent.
  Math. \textbf{120} (1995), no.~1, 61--79.

\bibitem{Heinonen-Koskela-CG}
\bysame, \emph{Quasiconformal maps in metric spaces with controlled geometry},
  Acta Math. \textbf{181} (1998), no.~1, 1--61.

\bibitem{Koranyi-Reimann1}
A.~Kor{\'a}nyi and H.~M. Reimann, \emph{Quasiconformal mappings on the
  {H}eisenberg group}, Invent. Math. \textbf{80} (1985), no.~2, 309--338.

\bibitem{Koranyi-Reimann2}
\bysame, \emph{Foundations for the theory of quasiconformal mappings on the
  {H}eisenberg group}, Adv. Math. \textbf{111} (1995), no.~1, 1--87.

\bibitem{LeDonne-Ottazzi}
E.~Le~Donne and A.~Ottazzi, \emph{Isometries between open sets of {C}arnot
  groups and global isometries of sub{F}insler homogeneous manifolds},
  Preprint, submitted (2012).

\bibitem{Magnani-Lipschitz}
V.~Magnani, \emph{Contact equations, {L}ipschitz extensions and isoperimetric
  inequalities}, Calc. Var. Partial Differential Equations \textbf{39} (2010),
  no.~1--2, 233--271.

\bibitem{nagelstwe}
A.~Nagel, E.~M. Stein, and S.~Wainger, \emph{Balls and metrics defined by
  vector fields. {I}. {B}asic properties}, Acta Math. \textbf{155} (1985),
  no.~1--2, 103--147.

\bibitem{Ottazzi-Hess}
A.~Ottazzi, \emph{Multicontact vector fields on {H}essenberg manifolds}, J. Lie
  Theory \textbf{15} (2005), no.~2, 357--377.

\bibitem{Ottazzi-Warhurst-1}
A.~Ottazzi and B.~Warhurst, \emph{Algebraic prolongation and rigidity of
  {C}arnot groups}, Monatsh. Math. \textbf{162} (2011), no.~2, 179--195.

\bibitem{Ottazzi_Warhurst}
\bysame, \emph{Contact and 1-quasiconformal maps on {C}arnot groups}, J. Lie
  Theory \textbf{21} (2011), no.~4, 787--811.

\bibitem{Pansu}
P.~Pansu, \emph{M\'etriques de {C}arnot-{C}arath\'eodory et quasiisom\'etries
  des espaces sym\'etriques de rang un}, Ann. of Math. (2) \textbf{129} (1989),
  no.~1, 1--60.

\bibitem{Reimann-Htype}
H.~M. Reimann, \emph{Rigidity of {$H$}-type groups}, Math. Z. \textbf{237}
  (2001), no.~4, 697--725.

\bibitem{Tanaka2}
N.~Tanaka, \emph{On differential systems, graded {L}ie algebras and
  pseudogroups}, J. Math. Kyoto Univ. \textbf{10} (1970), 1--82.

\bibitem{Warhurst-filiform}
B.~Warhurst, \emph{Contact and quasiconformal mappings on real model filiform
  groups}, Bull. Austral. Math. Soc. \textbf{68} (2003), no.~2, 329--343.

\bibitem{Warhurst-jet}
\bysame, \emph{Jet spaces as nonrigid {C}arnot groups}, J. Lie Theory
  \textbf{15} (2005), no.~1, 341--356.

\bibitem{Warhurst-free}
\bysame, \emph{Tanaka prolongation of free {L}ie algebras}, Geom. Dedicata
  \textbf{130} (2007), 59--69.

\bibitem{Warhurst-2}
\bysame, \emph{Contact and {P}ansu differentiable maps on {C}arnot groups},
  Bull. Aust. Math. Soc. \textbf{77} (2008), no.~3, 495--507.

\bibitem{Warner}
F.~W. Warner, \emph{Foundations of {D}ifferentiable {M}anifolds and {L}ie
  {G}roups}, Graduate Texts in Mathematics, vol.~94, Springer-Verlag, New York,
  1983, Corrected reprint of the 1971 edition.

\bibitem{Xie-filiform}
X.~Xie, \emph{Quasiconformal maps on model filiform groups}, arXiv:1308.3027
  (2013).

\bibitem{Yamaguchi}
K.~Yamaguchi, \emph{Differential systems associated with simple graded {L}ie
  algebras}, Progress in {D}ifferential {G}eometry, Adv. Stud. Pure Math.,
  vol.~22, Math. Soc. Japan, Tokyo, 1993, pp.~413--494.

\end{thebibliography}
\bibliographystyle{amsplain}
\end{document}